\documentclass[12pt, oneside]{amsart}
\usepackage{amsmath}
\usepackage{amssymb}
\usepackage{color}
\usepackage[usenames,dvipsnames,x11names,svgnames]{xcolor}
\usepackage[colorlinks=true,linkcolor=NavyBlue,citecolor=DarkGreen, urlcolor=blue]{hyperref}
\usepackage{amsthm}
\usepackage{amscd}
\usepackage{geometry}
\usepackage{mathrsfs}
\usepackage{dsfont}
\usepackage{mathtools}
\usepackage{eurosym}
\usepackage{booktabs}
\usepackage{float}
\usepackage{enumerate}

\newtheorem{thm}{Theorem}

\newtheorem{prop}{Proposition}
\newtheorem{lem}[prop]{Lemma}
\newtheorem{cor}[thm]{Corollary}

\theoremstyle{definition}
\newtheorem{defn}{Definition}
\newtheorem*{rem}{Remark}

\newtheorem*{ack}{Acknowledgments}

\newcommand{\mb}{\mathbb}
\newcommand{\mc}{\mathcal}

\newcommand{\ol}{\overline}

\newcommand{\leqs}{\leqslant }
\newcommand{\geqs}{\geqslant }
\newcommand{\wh}{\widehat}

\newcommand{\eps}{\varepsilon}
\newcommand{\1}{\mathds{1}}

\newcommand{\be}{\begin{equation*}}
\newcommand{\ee}{\end{equation*} }

\newcommand{\ben}{\begin{equation}}
\newcommand{\een}{\end{equation} }

\newcommand{\bs}{\begin{split}}
\newcommand{\es}{\end{split}}

\newcommand{\bmu}{\begin{multline*}}
\newcommand{\emu}{\end{multline*}}

\newcommand{\bmun}{\begin{multline}}
\newcommand{\emun}{\end{multline}}

\newtagform{Alph}[]()
\newtagform{roman}[]()
\newtagform{scroman}[][]

\begin{document}
\title[Siegel Zeros and small gaps]{Siegel Zeros and small gaps between zeros of the Riemann zeta function}
\author[A. Bondarenko]{Andriy Bondarenko}
\address{Department of Mathematical Sciences, Norwegian University of Science and Technology (NTNU), 7491 Trondheim, Norway.}
\email{andriybond@gmail.com} 

\author[W. Heap]{Winston Heap}
\email{winstonheap@gmail.com}

\begin{abstract}
On assuming the Riemann Hypothesis, we show that Siegel zeros imply the existence of gaps between the zeros of the Riemann zeta function less than $1/2$ the normalised length. Specifically, we show that an infinite family of Siegel zeros implies $\liminf_{n\to\infty}(\gamma_{n+1}-\gamma_n)\log(\gamma_n)/2\pi< 0.4733$ on RH. This refutes the existence of certain strong alternative hypotheses under these assumptions. Our arguments incorporate long Dirichlet polynomials of length $T^{17/14-\eps}$ into the Montgomery--Odlyzko method. 
\end{abstract}

\maketitle

\section{Introduction}

The distribution of the non-trivial zeros of the Riemann zeta function has long held interest due to its significant influence on arithmetic. The Riemann Hypothesis (RH) provides the ultimate resolution to the horizontal distribution, with partial results in the form of zero density estimates known. Yet, the vertical distribution is less well understood and seemingly much more nuanced. 

On average, the zero ordinates $\gamma$ are separated by length $\sim 2\pi/\log \gamma$ as $\gamma\to\infty$. Thus, a natural question regarding the vertical distribution is to ask whether smaller or larger gaps exist. In this regard, ordering the vertical zero ordinates by $\gamma_n$, define the quantities 
\[
\mu=\liminf_{n\to\infty}\frac{\gamma_{n+1}-\gamma_n}{2\pi/\log \gamma_n},
\qquad
\lambda=\limsup_{n\to\infty}\frac{\gamma_{n+1}-\gamma_n}{2\pi/\log \gamma_n}.
\] 
A longstanding conjecture states that $\mu=0$ and $\lambda=\infty$. The first non-trivial estimate $\mu<1<\lambda$ was claimed by Selberg whilst the current best results give, on RH, $\lambda>3.18$ due to Bui--Milinovich \cite{BM} and $\mu<0.50895$  due to a recent breakthrough of Inoue \cite{Inoue}. 

In this article, we are interested in small gaps. These have an interesting connection to arithmetic; specifically, Siegel zeros.   

\begin{defn}[Siegel zero] 
A \emph{Siegel zero} is a real number $\beta$ associated to a primitive quadratic Dirichlet character $\chi$ modulo $q$ such that $L(\beta,\chi)=0$ and 
\[
\beta= 1-\frac{1}{\mc{E}\log q}
\]
with $\mc{E}\geqs 3$. The quantity $\mc{E}$ is referred to as the \emph{quality} of $\beta$. We refer to an infinite family $(\beta_j,\chi_j,q_j,\mc{E}_j)$ with $q_j\to\infty$ as an \emph{exceptional sequence}.  
\end{defn}

It has long been known that Siegel zeros imply a very surprising rigidity in the vertical distribution of zeros of the Riemann zeta function. In particular, they are intimately related to gaps of length 1/2. In unpublished work (but see \cite{HBtalk}), Heath-Brown showed that an exceptional sequence would imply many long ranges where almost all normalised zeta zeros are half-integer separated -- this type of distribution has since become known as the Alternative Hypothesis. Later, Conrey--Iwaniec \cite{CI} showed that the existence of a Siegel zero implies that the number of zeta zeros up to height $T$ whose normalised gaps are less than $1/2-\eps$ is $\ll T(\log T)^{4/5}$. That is, almost no zeros are within $<1/2$ of each other provided the existence of a Siegel zero. 


Given this, along with the variety of techniques, gradual nature of progress, and extensive efforts made in achieving the current best bound $\mu<0.50895$, it is natural to suspect that $1/2$ represents a natural barrier due to the potential existence of Siegel zeros. Indeed, this sentiment has been expressed on many occasions \cite{C,GTTB, Inoue}. In this regard, we have the following surprising result.

\begin{thm}
\label{main thm}
Assume RH and that there exists an exceptional sequence of Siegel zeros with $\mc{E}=\inf_j \mc{E}_j$ sufficiently large. Then
\[
\mu<0.4733.
\]
\end{thm}

We emphasise that this is not in contradiction with the results of Heath-Brown or Conrey--Iwaniec as we are not establishing a quantitative proportionality result; we shall comment on this later. In the immediate, we note that our result has consequences for potential Alternative Hypotheses. Although these came about through supposing the existence of Siegel zeros and were approximate in some sense, they nevertheless do not require Siegel zeros to exist and could do so in potentially stronger forms, with no exceptional sets. Even these stronger forms have proved frustratingly difficult to refute, and so it seems reasonable to note the following.     

\begin{cor}Assume RH and suppose the existence of an exceptional sequence with $\mc{E}=\inf_j \mc{E}_j$ sufficiently large. Then an asymptotic distribution of the form 
\[
\frac{\gamma_{n+1}-\gamma_n}{2\pi/\log \gamma_n}\in \tfrac12 \mb{Z}_{\geqs 1}+o(1)
\] 
cannot exist. That is, in a strong alternative distribution $\tfrac12\mb{Z}+o(1)$, zeros cannot all exist in isolation and there must be some clustering on a normalised $o(1)$ level. 
\end{cor}

It seems surprising that not only do Siegel zeros not act as a barrier to $\mu<1/2$, but that they can themselves be used to establish this. However, on examining the methods for small gaps it turns out to be quite a natural assumption, as we now describe. We shall use the Montgomery--Odlyzko method \cite{MontgomeryOdlyzko}. It is well known that this has theoretical limitations to establishing $\mu<1/2$ -- in fact even $\mu<0.508$ \cite{IKT} -- however, crucially, such limitations assume that one is using so-called short Dirichlet polynomials. Our main innovation is to dispense with this condition and use long polynomials. 

The method relies on a comparison of the two integrals 
\[
\mc{I}_1=\int_T^{2T} (N(t+h)-N(t))|R(t)|^2dt, \qquad
\mc{I}_0=\int_T^{2T} |R(t)|^2dt
\]
where $N(t)$ denotes the number of zeros up to height $t$, $h=2\pi c/\log T$ and $R$ is some function to be chosen. If it can be shown that $\mc{I}_1>\mc{I}_0$ then there exists at least two zeros in $[T,2T]$, endpoints aside, within a normalised distance $<c$ of each other.  

A natural choice is to take 
\[
R(t)=\sum_{n\leqs L}\frac{\mu(n)}{n^{1/2+it}}
\] since this emulates $1/\zeta(\tfrac12+it)$ and hence becomes large around small gaps. From a more practical viewpoint, the property $\mu(p)=-1$ flips a sign in the prime Dirichlet polynomial approximation to $N(t+h)-N(t)$ exhibiting large \emph{positive} values in $\mc{I}_1$.  

The limitations of the method are computational; one must restrict the length $L\leqs T$ in order to have an approximate orthogonality and retain only diagonal contributions. 
We would like to take a larger $L$, however a serious obstacle here is that in order to handle off-diagonal terms, one must understand additive correlations of the coefficients:
\[
\sum_{n\leqs x}\mu(n)\mu(n+h),
\] 
and, moreover, achieve a power saving. For the M\"obius function this seems far out of reach currently. Nevertheless, we arrive at the following scheme which could hope to establish smaller gaps. 

\vspace{0.3cm}
\noindent $\bullet$ Construct a multiplicative function $f$ with the following properties:
\begin{enumerate}[(i)]
\item $f(p)=-1$ on sufficiently many primes.
\item The correlations 
\[
\sum_{n\leqs x}f(n)f(n+h),\qquad \sum_{km\leqs x}\Lambda(k)f(m)f(km+h)
\]
are computable with power saving error. 
\end{enumerate}

\vspace{0.3cm}

From this it is clear why Siegel zeros can help: choosing $f$ as an exceptional character gives the first condition whilst, in general, Dirichlet characters are very structured so there is hope to understand the correlations. This combination of properties is precisely the mechanism behind Heath-Brown's work on twin primes \cite{HB} and subsequent results \cite{Chinis,GM,TaoTeravainen}. Unlike the twin prime and Chowla conjectures though, we require the additional shifted prime correlations. This is where the majority of our work lies. Here, we introduce some new ingredients building on the Burgess/Vinogradov--type arguments of Kerr \cite{Kerr} which may be of independent interest. See Section \ref{rough sketch sec} for a rough sketch of our ideas.  


We close this introduction with some remarks and discussion on technicalities and possible improvements. Our results show that one need not be wary of Siegel zeros when trying to break the $\mu<1/2$-barrier. Most important is that one can use long Dirichlet polynomials.
Using the recent Chen--Gupta--Li \cite{ChenGuptaLi} zero-density estimate, our methods allow a length $L\approx T^{17/14}$. The Density Hypothesis for Dirichlet $L$-functions would allow  $L\approx T^{5/4}$, although this only results in a modest improvement to around $\mu<0.467$. It is possible that further refinements in the choice of coefficients could offer improvements, as in previous works on small gaps \cite{BMN, CGG, FW, P}, although we suspect that this would only offer modest gains also.

The recent improvement of Inoue \cite{Inoue} utilises a quadratic difference of $N(t)$ and thus in the long polynomial setting would require additional, more complicated, correlations. Since these will imply further restrictions on the length, we have not pursued this method although it is possible that improvements could still exist. 

Finally, 
in order to establish proportionality results, one must input higher moments, or similar such additional information -- e.g.\, see \cite{BGMM,CGGGHB}. These restrict the length much more seriously making bounds of similar quality harder to achieve. Of course, in this setting Siegel zeros become a genuine barrier by Conrey--Iwaniec's result. Nevertheless, it would be of interest to see what can be obtained using Siegel zeros. The current best result regarding a positive proportion of small gaps is due to Chirre--Gon\c{c}alves--de Laat \cite{ChirreGoncalvesdeLaat} at $0.6039$ with refinements due to Bui--Goldston--Milinovich--Montgomery \cite{BGMM}.

\begin{ack}The authors gratefully acknowledge the assistance of OpenAI's ChatGPT, whose  exploratory discussions, calculations and numerics helped stimulate the development of this paper.
\end{ack}

\section{The Montgomery--Odlyzko setup with long polynomials}

\subsection{Weight choices and the small gaps condition}
Our aim is to compute and then compare a weighted integral of 
\[
N_h(t):=N(t+\tfrac{h}{2})-N(t-\tfrac{h}{2})=\sum_{\gamma} \mathds{1}_{|t-\gamma|\leqs h/2}
\] 
where, here and throughout, 
\[
h=\frac{2\pi c}{\log T}
\]
with $c>0$. 


We choose our weights with desirable properties so let us describe these first. Throughout, we use the Fourier convention
\[
 \wh F(\xi)=\int_{\mb{R}}F(t)e^{-2\pi i t\xi}\,dt.
\]
We choose a nonnegative entire weight concentrated around $T$ with a smooth compactly supported Fourier
transform.  Fix $\sigma>0$ and a nonzero real even function
$\varphi\in C_c^\infty((-\sigma,\sigma))$, and let
\[
 \Phi(z)=\int_{\mb{R}}\varphi(\xi)e^{2\pi i z\xi}\,d\xi,
 \]
 and
\begin{equation} \label{eq:entire-weight}
 W_T(z)=\bigg(\frac{z^2+1/4}{T^2}\bigg)^B
 \bigg(\Phi\!\left(\frac zT-1\right)^2
       +\Phi\!\left(\frac zT+1\right)^2\bigg)
\end{equation}
with $B\in\mb{N}$  to be chosen. 
Then $W_T$ is even and entire, $W_T(t)\geqs0$ for $t\in\mb{R}$, and
\begin{equation}
 \mathrm{supp}\,\wh W_T\subseteq
 \left[-\frac{2\sigma}{T},\frac{2\sigma}{T}\right].
 \label{eq:weight-support}
\end{equation}
Moreover
\begin{equation}
 \wh W_T(0)=C_\Phi T+O_{\Phi,B}(T^{-1}),
 \label{eq:weight-mass}
\end{equation}
where
\[
 C_\Phi=\int_{\mb{R}}x^{2B}
 \{\Phi(x-1)^2+\Phi(x+1)^2\}\,dx>0.
\]
The precise choice of $\varphi$ is not so relevant as the constant $C_\Phi$ will appear in both the numerator and denominator of a final ratio and hence cancel. 

For now, let
\[
R(t)=\sum_{n\leqs L}\frac{r(n)}{n^{1/2+it}}
\]
where the coefficients and length are to be determined -- at this stage we assume real coefficients $r(n)\ll n^\eps$ and $L\leqs T^C$. We then consider the following two integrals 
\[
\mc{I}_1=\int_{T^{1-\eps}}^{T^{1+\eps}}N_h(t)|R(t)|^2W_T(t)dt,
\qquad
\mc{I}_0=\int_{T^{1-\eps}}^{T^{1+\eps}}|R(t)|^2W_T(t)dt.
\]

\vspace{0.3cm}
\noindent{$\bullet$ \emph{Small gaps criterion}:} If $\mc{I}_1>\mc{I}_0$ there exists a $t\in[T^{1-\eps},T^{1+\eps}]$ for which $N_h(t)>1$, thus giving a small gap of normalised size $\leqs c(1+\eps)$. 

\vspace{0.3cm}

To compute these integrals, we first extend them to $\mb{R}$. From the rapid decay of $\Phi$ and the trivial bounds $R(t)^2\ll L^{1+\eps}$, $N_h(t)\ll \log (2+|t|)$ the tail $|t|>T^{1+\eps}$ is
\[
\ll T^{\eps (2B-A)}L^{1+\eps'}
\]
for any given $A>0$. Choosing $A$ sufficiently large in terms of $B$, $\eps$ and $L$ then gives $O(T^{-C})$. For the region $|t|<T^{1-\eps}$, the pre-factor of $((t^2+1/4)/T^2)^B$ in $W_T$ gives the necessary decay on choosing $B$ large enough in terms of $\eps$. Then, accounting for the symmetry of the zeros, 
\[
\frac{\mc{I}_1}{\mc{I}_0}=\frac{I_1}{I_0}+O(T^{-C})
\]  
where 
\[
I_1=\int_\mb{R}N_h(t)|R(t)|^2W_T(t)dt,\qquad I_0=\int_\mb{R} |R(t)|^2W_T(t)dt. 
\]
Here we have assumed $I_0\gg 1$ when pulling it out of the denominator; needless to say, we will have this bound.  

\subsection{Application of the explicit formula} We now reduce $I_1$ to a sum involving primes via the explicit formula. Interchanging, we have 
\[
I_1=\sum_\gamma B_{T,h}(\gamma) 
\]
where
\begin{equation}\label{B}
B_{T,h}(u)=\int_\mb{R}\mathds{1}_{|u-t|<h/2}|R(t)|^2W_T(t)dt.
\end{equation}
Then, as a convolution, 
\[
\wh{B_{T,h}}(\xi)=\frac{\sin(\pi h\xi)}{\pi\xi}\sum_{m,n\leqs L}\frac{r(m){r(n)}}{\sqrt{mn}}\wh{W_T}\Big(\xi+\frac{\log(m/n)}{2\pi}\Big)
\]
and for complex $z$ we interpret 
\[
B_{T,h}(z)=\int_{-h/2}^{h/2}R(t+z)\ol{R}(t+z) W_T(t+z)dt
\]
where 
\[
 \ol{R}(z)=\sum_{n\leqs L}\frac{{r(n)}}{n^{1/2-iz}}.
\]
Since $\wh B$ has compact support we may apply the Guinand--Weil explicit formula to give, assuming RH, 
\begin{multline*}
I_1=\frac{1}{2\pi}\int_{\mb{R}}B_{T,h}(u)\bigg(\Re\frac{\Gamma'}{\Gamma}
 \left(\frac14+\frac{i u}{2}\right)-\log\pi\bigg)du 
 +B_{T,h}(i/2)+B_{T,h}(-i/2) \\
 -\frac1\pi\sum_{k\geqs 2}
 \frac{\Lambda(k)}{\sqrt k}\wh{B_{T,h}}\Big(\frac{\log k}{2\pi}\Big).
\end{multline*}

Now, for the polar terms we have 
\[
B_{T,h}(i/2)\ll \int_{-h/2}^{h/2} \left|\frac{(i/2+u)^2+1/4}{T^2}\right|^B |R(i/2+u)\ol{R}(i/2+u)|du
\ll \frac{L^{1+\eps}}{T^{2B}}\ll T^{-C}
\]
on taking $B$ large enough in terms of $L$ and similarly for $B(-i/2)$. For the Gamma factor integral, we use \eqref{B} and interchange integrals. By the decay of $W_T$, we may again restrict the $t$ integral to $T^{1-\eps}\leqs |t|\leqs T^{1+\eps}$ after applying Stirling's formula to estimate the inner integral by $\log(2+|t|)$. Then by Stirling's formula again
\[
\int_{\mb{R}}\mathds{1}_{|u-t|<h/2}\bigg(\Re\frac{\Gamma'}{\Gamma}
 \left(\frac14+\frac{i u}{2}\right)-\log\pi\bigg)du= h\log|t|+o(1).
\]
For $|t|$ in the current range, this is 
\[
= 2\pi c(1+O_1(\eps))+o(1)
\]
where $O_1$ denotes an implicit constant of absolute value $\leqs 1$. 
After applying this we may extend the $t$ integral to $\mb{R}$ at the cost of a negligible error again, giving in total the following. 

\begin{prop}
Assume RH. Then for any fixed $\eps>0$ we have 
\begin{multline*}
I_1=cI_0-\frac{2}{\pi}\sum_{\substack{m,n\leqs L\\k\geqs 2}}\frac{\Lambda(k)g_h(k)r(m)r(n)}{\sqrt{kmn}}\wh W_T\Big(\frac{\log(km/n)}{2\pi}\Big)
+O_1(\eps I_0)+o(I_0)+O_\eps(T^{-C})
\end{multline*}
where
\[
g_h(k)=\frac{\sin( \tfrac12h\log k)}{\log k}.
\]
\end{prop}

\subsection{Choice of resonator and proof of Theorem \ref{main thm}}
We now fix a choice of resonator, state our parameters, and give the required mean value estimates. We will then be able to conclude Theorem \ref{main thm}.

 Let $q$ be a fundamental discriminant, large, and set
\[
R(t)=\sum_{n\leqs L}\frac{\chi(n)G(n)}{n^{1/2+it}},
\]
where $\chi$ is the primitive quadratic character modulo $q$ and 
\[
G(n)=G_0\Big(\frac{\log n}{\log L}\Big)f(n)
\]
with $G_0$ a polynomial and $f$  a smooth function equal to $1$ on $[0,L/2]$ which decreases smoothly to zero on $[L/2,L]$ whereafter it remains zero. The derivatives of $f$  satisfy $f^{(j)}\ll L^{-j}$ and the unsmooth condition $n\leqs L$ can essentially be removed, although we retain it for emphasis. 

Fix $0<\delta<10^{-2}$ here and throughout and set
\begin{equation}
 T=q^{7/3+\delta},\qquad
 L=Tq^{1/2-\delta}=q^{17/6},\qquad
 \vartheta=\vartheta_\delta=\frac{\log L}{\log T}
 =\frac{17}{14+6\delta}.
 \label{eq:parameters}
\end{equation}
We then have the following mean value asymptotics. 

\begin{thm}
\label{denom thm}
Let 
\[
 I_0=\int_\mb{R} |R(t)|^2W_T(t)dt
\]
with the above choices of $T$, $L$ and $R(t)$. Then
\begin{equation}
 I_0=(1+o(1))D_G\wh W_T(0)\frac{\phi(q)}q\log L
 \label{eq:I0-asymptotic}
\end{equation}
where
\begin{equation}
 D_G=\int_0^1G_0(x)^2\,dx.
 \label{eq:DG}
\end{equation}
\end{thm}

\begin{thm}\label{num thm}
Let 
\[
J=\sum_{\substack{m,n\leqs L\\k\geqs 2}}\frac{\Lambda(k)g_h(k)\chi(m)\chi(n)G(m)G(n)}{\sqrt{kmn}}\wh W_T\Big(\frac{\log(km/n)}{2\pi}\Big)
\]
so that $I_1=cI_0-(2/\pi)J+O_1(\eps I_0)+o(I_0)$.

Suppose that $\chi$ is an exceptional character with associated quality $\mc{E}$. Then 
\[
J=-
\Big(1+o(1)+O\Big(\frac{1}{\sqrt{\log\mc{E}}}\Big)\Big)\wh W_T(0)\frac{\phi(q)}q\log L
 \int_0^1 C_G(u)
 \frac{\sin(\pi c\vartheta u)}u\,du.
\]
where 
\[ C_G(u)=\int_0^{1-u}G_0(x)G_0(x+u)\,dx\]
and $\vartheta=\log L/\log T$.
\end{thm}

\begin{proof}[Proof of Theorem \ref{main thm}]
Applying these two asymptotics gives
\begin{equation}\label{ratio}
\frac{\mc{I}_1}{\mc{I}_0}= c+ \frac{2}{\pi D_G}\int_0^1 C_G(u)
 \frac{\sin(\pi c\vartheta u)}u\,du+O(\eps)+O\Big(\frac{1}{\sqrt{\log\mc{E}}}\Big)+o(1).
\end{equation}
We now fix
\[
 \delta=10^{-6},\qquad
 \vartheta=\frac{17}{14+6\delta},
 \qquad
 c=0.473275,
\]
and choose
\[
 G_0(x)=1-\frac{63}{125}\left(x-\frac12\right)^2
\]
which is positive on $[0,1]$.  A computer calculation then shows that the right hand side of \eqref{ratio} is $>1$ for $\eps$ sufficiently small and $\mc{E}$ sufficiently large.
\end{proof}


\section{Auxiliary results}

During the proof, we shall make use of two additional results. The first is a classical zero density estimate. Let $N(\sigma,H,\psi)$ denote the number of zeros of $L(s,\psi)$ in the region $\{\rho=\beta+i\gamma: \beta>\sigma, |\gamma|\leqs H\}$. Then we have the following.

\begin{prop}
\label{prop:density}
For $1/2\leqs\sigma<1$, $H\geqs1$, and every $\eps>0$,
\begin{equation}
 \sideset{}{^*}\sum_{\psi\bmod q}N(\sigma,H,\psi)
 \ll_{\eps}
 (qH)^{7(1-\sigma)/3+\eps}
 \label{eq:density}
\end{equation}
where the sum is restricted to primitive characters modulo $q$.
\end{prop}

This is due to Chen--Gupta--Li \cite{ChenGuptaLi} which utilises the machinery of Guth--Maynard \cite{GM}. The exponent $7/3$ improves on Huxley's classical exponent of $12/5$ \cite{HuxleyLVDP}.

Our second input is a result of Heath-Brown \cite{HB} (Lemma 3) which concisely captures the fact that $\chi(p)=-1$ almost always in the presence of a Siegel zero. 

\begin{prop}
\label{Heath-Brown prop}
Let $\beta=1-(\mc{E}\log q)^{-1}$ be a real zero of $L(s,\chi)$, with
$\mc{E}\geqs 3$.  Then
\begin{equation}
 \sum_{\substack{p\leqs q^{500}\\ \chi(p)=1}}
 \frac{\log p}{p}
 \ll \frac{\log q}{\sqrt{\log\mc{E}}}
 \label{eq:bad-primes}
\end{equation}
where the implied constant is absolute.
\end{prop}

\section{The denominator $I_0$: Proof of Theorem \ref{denom thm}}

In this section we consider the mean square
\[
I_0=\int_\mb{R}|R(t)|^2W_T(t)dt
\]
where 
\[
R(t)=\sum_{n\leqs L}\frac{\chi(n)G(n)}{n^{1/2+it}}.
\]
Since $q$ is a fundamental discriminant it takes the form
\begin{equation}
 q=2^\nu Q,\qquad \nu\in\{0,2,3\},\qquad Q\text{ odd and square-free}.
 \label{eq:fundamental-q}
\end{equation}
Let us also recall our parameter selections
\[
 T=q^{7/3+\delta},\qquad
 L=Tq^{1/2-\delta}=q^{17/6}.
\]
These exact choices are not crucial for this section and could be taken with greater flexibility -- they are essentially determined by our argument for the numerator, although we state them here for clarity. 

Since the length of the resonator is greater than the length of integration, in order to compute the off-diagonal terms we must understand the correlations $\sum\chi(m)\chi(m+r)$. The weights in the off-diagonal sums will be of variation at most $q^\eps$; so for smooth $V$ we define
\[
\|V\|_{\mathrm{BV}}=\|V\|_\infty+\int_\mb{R}|V'(x)|dx.
\] 
We then have the following general result.

\begin{lem}
\label{lem:completion}
Let $k,r$ be nonzero integers and let $V$ be a smooth weight supported on a dyadic
interval $m\asymp M$ with $\|V\|_{\mathrm{BV}}\ll q^{\eps}$.  Then
\begin{equation}
 \sum_mV(m)\chi(m)\chi(km+r)
 \ll_{\eps}q^{\eps}
 \left\{\frac{M(r,Q)}q+\sqrt{q(r,Q)}\right\}.
 \label{eq:completion}
\end{equation}
\end{lem}

\begin{proof}
We first complete the sum by writing it as 
\[ 
\sum_{a\bmod q}\chi(a)\chi(ak+r)\sum_{n\equiv a\bmod q}V(n)
\]
and applying orthogonality of additive characters. This gives 
\begin{equation}\label{completion}
\frac{S_q(k,r;0)}{q}\sum_n V(n)+\frac{1}{q}\sum_{b=1}^{q-1}S_q(k,r;b)\sum_n V(n)e(bn/q)
\end{equation}
where
\[
 S_q(k,r;b)=\sum_{a\bmod q}
 \chi_q(a)\chi_q(ak+r)e(-ab/q).
\]
Here, the subscript $q$ in $\chi$ denotes its modulus -- to be used for emphasis shortly.
In the first term of \eqref{completion} we have $\sum V=\int V+O(q^\eps)\ll q^\eps M$ whilst, by partial summation, the second term is 
\[
\ll \|V\|_{\mathrm{BV}}\sum_{b=1}^{q-1} \frac{|S_q(k,r;b)|}{q\|b/q\|}
\]
since $\sum_{n\leqs t} e(bn/q)\ll 1/\|b/q\|$ where $\|\cdot\|$ denotes the distance to the nearest integer. Thus, it remains to bound the complete sum $S_q(k,r;b)$ pointwise in $b$.

Writing $q=2^\nu p_1\cdots p_j$, by the Chinese Remainder Theorem we have the factorisations $\chi_q=\chi_{2^\nu}\chi_{p_1}\cdots \chi_{p_j}$ and 
\[
S_q(k,r;b)=S_{2^\nu}(k,r;b_\nu)\prod_{i=1}^jS_{p_i}(k,r;b_i)
\]
 for certain integers $b_i$ (all equal to zero if $b=0$).
When $b=0$, at an odd prime $p\mid q$, if $p|k$ the local factor is zero so we may assume otherwise. If $p\mid r$ it is $\chi_p(k)(p-1)$
 since $\chi_p$ is quadratic. When $p\nmid r$, we map $a\mapsto ar$ giving the sum 
\[
\sum_{a\bmod p}\chi_p(a)\chi_p(ak+1)=\sum_{\substack{a=1}}^{p-1}\chi_p(k+\bar{a})
=\sum_{\substack{a=1}}^{p-1}\chi_p(a+k)=-\chi_p(k)
\]
after adding in the term $a=p$ to complete the sum. Here, $\bar a$ is the multiplicative inverse of $a$ modulo $p$.  
  
  For nonzero $b_j$, a Gauss
sum or the Weil bound for a quadratic polynomial gives $O(\sqrt p)$, except
in the degenerate case already accounted for by a common divisor.  
We therefore have 
\begin{equation} \label{eq:complete-hybrid}
 |S_q(k,r;0)|\ll q^{\eps}(r,Q),
 \qquad
 |S_q(k,r;b)|\ll q^{1/2+\eps}(r,b,Q)^{1/2}\qquad(b\ne0).
\end{equation}
Summing over $b$ and using $(r,b,Q)\leqs (r,Q)$ and $\sum_{b\leqs q-1}(q\|b/q\|)^{-1}\ll \log q$ gives the result.
\end{proof}

\begin{proof}[Proof of Theorem \ref{denom thm}]
Expanding the square gives
\begin{equation}
 I_0=\sum_{m,n\leqs L}\frac{\chi(m)\chi(n)G(m)G(n)}{\sqrt{mn}}
 \wh W_T\!\left(\frac{\log(n/m)}{2\pi}\right).
 \label{eq:I0-expand}
\end{equation}
The diagonal is
\[
 \wh W_T(0)\sum_{\substack{n\leqs L\\(n,q)=1}}
 \frac{G(n)^2}{n}.
\]
Inserting $\1_{(n,q)=1}=\sum_{d\mid(n,q)}\mu(d)$ and writing $n=dm$, this sum is
\[
\sum_{d|q}\frac{\mu(d)}{d}\sum_{m\leqs L/d} \frac1mG_0\!\left(\frac{\log dm}{\log L}\right)^2 f(dm)^2. 
\]
Replacing $f(dm)^2$ by 1 in the inner sum introduces a logarithmic sum over $L/2d<m\leqs L/d$ which is $O(1)$. Then by Euler--Maclaurin summation, the inner sum is 
\[
\int_1^{L/d}G_0\!\left(\frac{\log dx}{\log L}\right)^2\frac{dx}{x}+O(1)
\]
which, after some basic substitutions and estimates, is $D_G\log L+O(\log d)$. Then since $\sum_{d|q}\mu(d)/d=\phi(q)/q$ we acquire the main term in \eqref{eq:I0-asymptotic}.
The total error is bounded by
\[
 \sum_{d\mid q}\frac{1+\log d}{d}
 \ll \prod_{p\mid q}\bigg(1+\frac1p\bigg)
 \bigg(1+\sum_{p\mid q}\frac{\log p}{p+1}\bigg)
 \ll(\log\log q)^2
\]
which is negligible compared to $(\phi(q)/q)\log L$.

For the off-diagonals, put $n=m+r$ and divide $m$ dyadically via a smooth partition of unity $\sum_M \omega(m/M)=1$ with $M\ll L$. We then acquire $\ll \log T$ sums of the form 
\[
 \frac{T}{M}\sum_{r\neq 0}\sum_{m\asymp M}\chi(m)\chi(m+r)V(m)
\]
with 
\[
V(m)=\frac{G(m)G(m+r)}{M^{-1}\sqrt{m(m+r)}}\frac{1}{T}\wh W_T\!\left(\frac{\log(1+r/m)}{2\pi}\right)\omega(m/M)
\]
Note that $\|V\|_{\mathrm{BV}}\ll 1$ since $\wh W_T\ll_\Phi T$. Also, from the support
conditions of $\wh W$ we may restrict 
$|r|\ll_\Phi M/T$. Then, Lemma~\ref{lem:completion} gives, for one block,
\begin{align*}
 \frac{T}{M}\sum_{0<|r|\ll M/T}
 \left|\sum_{m\asymp M}\chi(m)\chi(m+r)V(m)\right|
 &\ll q^\eps\left(\frac Mq+\sqrt q\right)
\end{align*}
after using 
\begin{equation}
 \sum_{0<|r|\leqs R}(r,Q)^\alpha\ll_{\eps}Rq^\eps
 \qquad(\alpha=1/2\text{ or }1),
 \label{eq:gcd-average}
\end{equation}
which follows by expanding over divisors of $Q$ and applying the divisor bound $d(q)\ll q^\eps$.
After summing the dyadic blocks, the off-diagonal is
\begin{equation*}
 \ll q^\eps\left(\frac Lq+\sqrt q\log(2L)\right)=o(T)
 \label{eq:I0-off}
\end{equation*}
on recalling that $L/q=q^{11/6}$ and $T=q^{7/3+\delta}$.
  This completes the proof.
\end{proof}

\section{The numerator $J$: diagonal terms}

Our goal in this section and the following is to understand
\[
J=\sum_{\substack{m,n\leqs L\\k\geqs 2}} \frac{\Lambda(k)g_h(k)\chi(m)\chi(n)G(m)G(n)}{\sqrt{kmn}}
\wh W_T\!\left(\frac{\log(n/km)}{2\pi}\right).
\]
We write 
\[
J=
\mc{D}+\mc{OD}
\]
where $\mc{D}$ is the sum with terms $km=n$ and $\mc{OD}$ is the remaining off-diagonal sum. We then have the following. 

\begin{prop}
\label{prop:diagonal}
Assume that $\chi$ is an exceptional character of quality $\mc{E}$. Then 
\[
\mc{D}=-
\Big(1+o(1)+O\Big(\frac{1}{\sqrt{\log\mc{E}}}\Big)\Big)\wh W_T(0)\frac{\phi(q)}q\log L
 \int_0^1 C_G(u)
 \frac{\sin(\pi c\vartheta u)}u\,du.
\]
where 
\[ C_G(u)=\int_0^{1-u}G_0(x)G_0(x+u)\,dx\]
and $\vartheta=\log L/\log T$.
\end{prop}

\begin{prop}\label{off diag prop}
We have 
\[
\mc{OD}=o\Big(T\frac{\phi(q)}{q}\log T\Big).
\]
\end{prop}
Combining these two propositions gives Theorem \ref{num thm}. 
As usual, the diagonal terms are relatively simple to deal with so we consider these first. 
We prove Proposition \ref{off diag prop} in the next section.

\begin{proof}[Proof of Proposition \ref{prop:diagonal}]
Taking $km=n$ we have 
\begin{align*}
 \mc{D}
 =
 &\wh W_T(0)
 \sum_{km\leqs L}
 \frac{\Lambda(k)\chi(k)\chi(m)^2 G(m)
 G(km)}{km}
 \frac{\sin(\pi c\log k/\log T)}{\log k}.
 \label{eq:prime-diagonal}
\end{align*}

The contribution from prime powers is negligible: the sum over $m$ is $\ll \log T$ and using $|\sin x|\leqs |x|$ along with the bound $\sum_{p, a\geqs 2}\log p/p^a=O(1)$ gives $O(T)$ in total since $\wh W_T(0)\asymp T$.  
Then after restricting to $k=p$, for the inner sum, arguing as before we find
\begin{equation}
 \sum_{\substack{m\leqs L/p\\(m,q)=1}}
 \frac{G(m)
 G(pm)}{m}
 =C_G\big(\tfrac{\log p}{\log L}\big)\frac{\phi(q)}q\log L
 +O((\log\log q)^2).
 \label{eq:inner-m-diagonal}
\end{equation}

Now, the prime number theorem and partial summation give
\begin{equation}
 \sum_{p\leqs L}\frac1p
 C_G\big(\tfrac{\log p}{\log L}\big)
 \sin\big(\tfrac{\pi c\log p}{\log T}\big)
\sim
 \int_0^1C_G(u)\frac{\sin(\pi c\vartheta u)}u\,du
 \label{eq:PNT-logscale}
\end{equation}
whilst the accumulated error from \eqref{eq:inner-m-diagonal} is $ O(T(\log\log q)^2)$, again since $\wh W_T(0)\asymp T$ and using $|\sin x|\leqs |x|$ to bound the outer sum over primes by $O(1)$. This error is negligible compared with $T(\phi(q)/q)\log L$.
It then remains to justify replacing $\chi(p)$ by $-1$.  By Proposition \ref{Heath-Brown prop}  we have 
\begin{equation*}
 \sum_{\substack{p\leqs q^{500}\\ \chi(p)\ne-1}}
 \frac{\log p}{p}
 \ll \frac{\log q}{\sqrt{\log\mc{E}}}+\log\log q.
\end{equation*}
and hence
\begin{equation*}
 \sum_{\substack{p\leqs L\\\chi(p)\ne-1}}
 \frac1p\left|\sin\!\left(\frac{\pi c\log p}{\log T}\right)\right|
 \ll\frac{\log q}{\log T\sqrt{\log\mc{E}}}+o(1)=O\Big(\frac{1}{\sqrt{\log \mc{E}}}\Big).
 \label{eq:bad-prime-error}
\end{equation*}
Here $L=q^{17/6}<q^{500}$.  This proves
Proposition \ref{prop:diagonal}.
\end{proof}

\section{The numerator: off-diagonal terms -- proof of Proposition \ref{off diag prop}}

\subsection{Sketch of argument}\label{rough sketch sec}
We first give a rough sketch of our arguments. Ignoring the smooth weights $g$ and $G$, we are required to understand the correlation sums
\[
\sum_{\substack{km\leqs x\\n-km=r}}\Lambda(k)\chi(m)\chi(n).
\]
If we can demonstrate a power saving error $O(x^\theta)$ we can take the length $L$ of the resonator larger than $T$. Indeed, after some work, the error for the correlation problem translates, more or less, directly into an error $O(L^\theta)$ for the full off-diagonals, allowing $L=T^{1/\theta-\eps}$. 

Now, when computing these correlations we first distinguish which of $k,m$ is smaller, lest we have an outer sum of intractable length $x$. This gives inner sums of two different forms:
\[
\sum_{k<m\leqs x/k}\chi(m)\chi(km+r)
\]
and, when $m$ is the smaller of the two,
\[
\sum_{m<k\leqs x/m}\Lambda(k)\chi(km+r). 
\]
The first sum can be handled with Lemma \ref{lem:completion} giving satisfactory bounds. Accounting for the outer sum of length $\sqrt{x}$ then gives a power saving error for a fairly wide range of $q$. 

The second sum is a $GL(1)$ shifted prime problem and hence should be tractable with known techniques (since the $GL(2)$ version, a.k.a. the Titchmarsh divisor problem, is well studied). Indeed, this problem has been considered by several authors beginning with Vinogradov and most recently Kerr \cite{Kerr} (see references therein for intermediate works). For composite $q$, Kerr's result gives power savings in certain ranges of $q,k$, but in our setting becomes inefficient when $k$ is too large in terms of $q$. One can reduce this inefficiency by considering a more general hyperbolic decomposition $\mathds{1}_{m\leqs k^\phi}+\mathds{1}_{m>k^\phi}$, say, 
 but this still falls just short. Such inputs allow $L\approx T^{1.02}$ whereas to beat $\mu<1/2$ with a M\"obius type resonator one needs $L\approx T^{1.08}$.  

For larger $k$ we add a new component which, rather than using Burgess-type character sum moment bounds as in previous works, begins by separating the variables in $km+r$ via Fourier inversion. This introduces sums over all characters modulo $q$, however this doesn't entail much loss when $k$ is large since one can efficiently handle the resulting sums 
\[
\sum_{\psi\bmod q}\bigg|\sum_{k\asymp K} \Lambda(k)\psi(k)\bigg|
\]   
via zero density estimates. These give the strong bound $\ll q^\eps K$ when $K\gg q^{7/3+\eps}$ (using the Chen--Gupta--Li density estimate). The sums over $m$ acquired from the Fourier inversion step are very structured and involve generalised Jacobi sums which typically have squareroot cancellation. 

In total then, given these three ingredients our arguments split into three regimes with the following allied methods: 
\begin{itemize}
\item Small $k$ -- bounds for correlation sums $\chi(m)\chi(km+r)$.
\item Intermediate $k$ -- second moment estimate for shifted character sums.
\item Large $k$ -- separation of variables in $\chi(km+r)$, zero-density estimates for the sum in $k$ and Jacobi sum formulae for the sum in $m,r$.
\end{itemize}
  
\subsection{Initial reductions} The exact off-diagonals are given by 
\[
\mc{OD}
=
\sum_{\substack{m,n\leqs L, r\in\mb{Z}\backslash\{0\},\, k\ll L:\\n-km=r}} \frac{\Lambda(k)g_h(k)\chi(m)\chi(n)G(m)G(n)}{\sqrt{kmn}}
\wh W_T\!\left(\frac{\log(n/km)}{2\pi}\right).
\]

We treat the case of $r>0$, the argument for $r<0$ being similar.
We first apply a dyadic partition of unity of the form $\sum_{X}\omega(x/X)^2=1$ in each variable $k,m,r$ where $\omega$ is a smooth function supported on $[1,2]$ and where $X$ varies over a dyadic sequence, $2^j$ say, with upper bound $\ll L$. This gives a sum of $(\log T)^{O(1)}$ terms of the shape
\begin{equation}\label{OD term}
\frac{T}{KM}\sum_{k,m, r}\Lambda(k)\chi(m)\chi(km+r)F(k,m,r)
\end{equation}
where 
\begin{multline*}
F(k,m,r)
=
\omega(k/K)^2\omega(m/M)^2\omega(r/R)^2\frac{g_h(k)G(m)G(km+r)}{(KM)^{-1}\sqrt{km(km+r)}}
\\
\times\frac{1}{T}
\wh W_T\!\left(\frac{\log(1+r/km)}{2\pi}\right).
\end{multline*}
We utilise $\omega^2$ here, rather than $\omega$, as it will be useful shortly to keep a smoothing in the sum as well as in the weight function.  
Note that from the support of $\wh W$ we have the restrictions
\[
R\ll \frac{KM}{T},\qquad KM\ll L;
\]
the second following from the first and the fact that $KM+R\ll L$. 


We now separate variables. Write
\[
KM(km(km+r))^{-1/2}=\frac{KM}{km}\Big(1+\frac{r}{km}\Big)^{-1/2}
\]
and for the $G_0(\log(km+r)/\log L)$ term we expand the polynomial $G_0$, write 
\[
\log(km+r)=\log k+\log m+\log\Big(1+\frac{r}{km}\Big)
\]
 in each monomial, 
and then expand the powers of this by the multinomial theorem. In this way we find that  
\eqref{OD term} can be written as a sum of $O_G(1)$ terms of the shape
\[
\frac{T}{KM}\sum_{k,m, r}a_k \Lambda(k)b_m\chi(m)\chi(km+r)\omega(k/K)\omega(m/M)\omega(r/R)\mc{F}(k,m,r)
\]
where $a_k$ is $g_h(k)$ times $K/k$ and some power of $\log k$, $b_m$ is $G(m)$ times $M/m$ and some power of $\log m$, and 
\[
\mc{F}(k,m,r)=\omega(k/K)\omega(m/M)\omega(r/R)\frac{f(km+r)\log^j(1+r/km)}{(1+r/km)^{1/2}}\cdot \frac1T\wh W_T\!\left(\frac{\log(1+r/km)}{2\pi}\right)
\]
for some $j\geqs 0$. Note that, thus far, $|a_k|,|b_m|\ll q^\eps$. 

We now separate variables in $\mc{F}$ via an inverse Mellin transform. 
This gives 
\[
\mc{F}(k,m,r)=\frac{1}{(2\pi i )^3}\int_{(c)^3} \widetilde{\mc{F}_0}(s_1,s_2,s_3)
\Big(\frac{K}{k}\Big)^{s_1}
\Big(\frac{M}{m}\Big)^{s_2}
\Big(\frac{R}{r}\Big)^{s_3}
ds_1ds_2ds_3
\] 
with $c>0$ where $\int_{(c)}=\int_{c-i\infty}^{c+i\infty}$ and 
\begin{multline*}
\widetilde{\mc{F}_0}(s_1,s_2,s_3)
=
\int_{\mb{R}_+^3}\omega(x)\omega(y)\omega(z)
\frac{f(KMxy+Rz)\log^j(1+\tfrac{Rz}{KMxy})}{(1+\tfrac{Rz}{KMxy})^{1/2}} 
\\
\times 
\frac1T\wh W_T\bigg(\frac{\log(1+\tfrac{Rz}{KMxy})}{2\pi}\bigg)
x^{s_1-1}y^{s_2-1}z^{s_3-1}dxdydz. 
\end{multline*}
From the compact support of $\omega$ we see $\widetilde{\mc{F}_0}$ is an entire function. Recalling that $R/KM\ll T^{-1}$ and that $f^{(j)}(x)\ll 1/L^j$, integration by parts shows that 
\[
\widetilde{\mc{F}_0}(s_1,s_2,s_3)
\ll_{j_1,j_2,j_3}
(1+|t_1|)^{-j_1}(1+|t_2|)^{-j_2}(1+|t_3|)^{-j_3}.
\]
Hence $\mc{F}$ is given by an absolutely convergent integral. We then truncate the $s_j$ integrals at height $q^\eps$. This gives a negligible error $O(T^{-C})$ by the rapid decay. The purpose of this move is so that the resultant weights after interchanging summation and integration have derivatives bounded by $q^\eps$ -- e.g. differentiating $(K/x)^{s_1}$ with respect to $x$ gives a factor of $-s_1$ which is then $\ll q^\eps$.    

Choosing $c=\eps$ and pushing the sum through the integral we acquire, after relabelling $a_k,b_m$ and renormalising coefficients, an inner sum of the form
\[
\frac{q^\eps T}{KM}\sum_{k,m, r}a_k \Lambda(k)b_mc_r\chi(m)\chi(km+r)
\]
where $a,b,c$ are complex valued 1-bounded smooth functions supported on $[K,2K]$, $[M,2M]$, $[R,2R]$, respectively, whose $j$th order derivatives are $\ll q^\eps$. Strictly speaking, the $q^\eps$ out front should be a normalisation factor $\|a,b,c\|_\infty$, which happens to satisfy the bound $\ll q^\eps$, however in the interests of brevity we write things in this simpler way. We finally remark that on writing $\log k=\log K+\log(k/K)$ we may express $a_k=V(k/K)$ for some smooth function $V(x)$ supported on $[1,2]$ satisfying the bounds $\|V^{(j)}\|_\infty\ll q^\eps$. This form will be preferable at later points.

We are thus required to show that the sum 
\[
E(K,M,R):=\sum_{k}a_k\Lambda(k)\sum_{m, r}b_mc_r\chi(m)\chi(km+r)
\]
has a fixed power saving, $q^{\eta}$ say, over $KM$. From this it will follow, after integrating, and then summing the $\ll (\log T)^{O(1)}$ partitions,  that
\[
\mc{OD}\ll \frac{q^{\eps}T}{q^\eta}=o\Big(T\frac{\phi(q)}{q}\log T\Big)
\]
on taking $\varepsilon$ sufficiently small in terms of $\eta$, thus giving Proposition \ref{off diag prop}.

Let us move on to proving this. 
We recall our parameter choices and a useful consequence: for a fixed $\delta>0$,
 \[
 T=q^{7/3+\delta},\qquad
 L=Tq^{1/2-\delta}=q^{17/6}, \qquad
 R\ll q^{1/2-\delta}
\]
since $R\ll KM/T\ll L/T$. 
Set $\eta=\delta/10$. The three ranges of $K$ we consider are given as follows: 
\begin{align}
 \text{I: }&K\leqs Tq^{-1/2-\eta},
 \label{eq:range-I}\\
 \text{II: }&Tq^{-1/2-\eta}<K\leqs Tq^{-\eta},
 \label{eq:range-II}\\
 \text{III: }&K\geqs q^{7/3+2\eta}.
 \label{eq:range-III}
\end{align}
Note ranges II and III overlap because
$7/3+2\eta<7/3+\delta-\eta$.

\subsection{Range I: $K\leqs Tq^{-1/2-\eta}$} Applying Lemma \ref{lem:completion} to the sum over $m$, \eqref{eq:gcd-average} for the sum over $r$ and estimating the sum over $k$ trivially we have 
\begin{align*}
\frac{E(K,M,R)}{KM}
\ll &
 \frac{1}{KM}\cdot q^\eps KR\bigg(\frac{M}{q}+\sqrt{q}\bigg) 
 \\
 \ll & q^\eps\bigg(\frac{R}{q}+\frac{K\sqrt{q}}{T}\bigg)
\end{align*}
using $R\ll KM/T$ for the second term. Then since $R\ll L/T= q^{1/2-\delta}$ always and $K\ll Tq^{-1/2-\eta}$ in this range, we acquire 
\begin{equation}\label{I}
\frac{E(K,M,R)}{KM}
\ll
 q^{-\eta/2}
 \end{equation} 
 in total for this range of $K$.  

\subsection{Range II: $Tq^{-1/2-\eta}<K\leqs Tq^{-\eta}$} Since $T=q^{7/3+\delta}$  we have $K\gg q$ in this range.  Also, 
$R/M\ll K/T\leqs q^{-\eta}$ and so we have $R\leqs M$ for sufficiently large $q$.  

Now,
$\chi(m)=0$ unless $(m,q)=1$ and hence in the sum over $m$ we may write
\begin{equation*}
 \chi(m)\chi(km+r)=\chi(k+r\bar m)
\end{equation*}
since $\chi(m)^2=1$ where $\bar m$ is the multiplicative inverse of $m\bmod q$. Then on swapping summation orders and retaining the condition $(m,q)=1$ in $b_m$, we have 
\begin{align*}
E(K,M,R)
= &
\sum_{m,r}b_mc_r\sum_k a_k\Lambda(k)\chi(k+r\bar m)
\\
= &
\sum_{x\bmod q}\sum_{\substack{m,r\\ r\bar m\equiv x\bmod q}}b_mc_r
\sum_k a_k\Lambda(k)\chi(k+x)
\end{align*}
after grouping terms according to $r\bar m$. Then by Cauchy--Schwarz, this is 
\begin{equation}\label{CS}
\ll 
\bigg(\sum_{x\bmod q}\bigg|\sum_{\substack{m,r\\ r\bar m\equiv x\bmod q}}b_mc_r\bigg|^2\bigg)^{1/2}
\bigg(\sum_{x\bmod q}\bigg|\sum_{k}a_k\Lambda(k)\chi(k+x)\bigg|^2\bigg)^{1/2}.
\end{equation}
For these sums we have the following two lemmas. 

\begin{lem}
Let $1\leqs R\leqs M\leqs q^C$.  For arbitrary coefficients
$|b_{r,m}|\leqs1$, supported on $r\asymp R$, $m\asymp M$ and $(m,q)=1$, put
\[
 W_x=\sum_{\substack{r,m\\r\bar m\equiv x\bmod q}}b_{r,m}.
\]
Then
\begin{equation}
 \sum_{x\bmod q}|W_x|^2
 \ll
 q^\eps\bigg(RM+\frac{R^2M^2}{q}\bigg).
 \label{eq:energy}
\end{equation}
\end{lem}

\begin{proof}
This is essentially an additive divisor problem. On discarding the coefficients, the left side counts solutions of
\begin{equation*}
 r_1m_2-r_2m_1=\ell q
\end{equation*}
where $\ell$ varies over $\mb{Z}$ as the support conditions allow.
Writing $\ell q =h$ gives $O(1+RM/q)$ possible values of $h$. We suppose $h\neq 0$; the case $h=0$ can be handled similarly to the following. Fixing $r_1,r_2,h$ in an outer sum, we see that there are no solutions to $ r_1m_2-r_2m_1=h$ if the gcd $g=(r_1,r_2)$ does not divide $h$. Organising the sum in terms of $g$ gives
\[
\sum_h\sum_{g|h}\sum_{\substack{r_j\asymp R/g\\ (r_1,r_2)=1}}\sum_{\substack{m_j\asymp M\\r_1m_2-r_2m_1=h/g}}1
\]

By Bezout's Theorem, there exists a solution to the equation $r_1m_2-r_2m_1=1$, given by $m_1=m_{1,0}$, $m_2=m_{2,0}$, say. The general solutions to $r_1m_2-r_2m_1=h/g$ then take the parametrised form 
\[
m_1=m_{1,0}(h/g)+nr_1,\qquad m_2=m_{2,0}(h/g)+nr_2
\] 
with varying $n\in\mb{Z}$. The support conditions then dictate $n\ll gM/R$ giving in total
\[
\ll \sum_h\sum_{g|h}\sum_{\substack{r_j\asymp R/g\\ (r_1,r_2)=1}}gM/R
\ll RM \sum_h\sum_{g|h}\frac{1}{g}
\ll q^\eps RM\Big(1+\frac{RM}{q}\Big).
\]
\end{proof}

\begin{lem}
\label{lem:translate-L2}
For arbitrary $b_k$ supported on $k\asymp K$,
\begin{equation}
 \sum_{x\bmod q}
 \left|\sum_{k\asymp K}b_k\chi(k+x)\right|^2
 \ll_\eps q^\eps(q+K)\sum_{k\asymp K}|b_k|^2.
 \label{eq:translate-L2}
\end{equation}
\end{lem}

\begin{proof}
Expanding the square and swapping summations gives an inner sum satisfying 
\[
 \left|\sum_{x\bmod q}\chi(x+k)\chi(x+\ell)\right|
 \ll (k-\ell,q).
\]
This follows by the similar arguments to Lemma \ref{lem:completion}: after mapping $x+\ell\mapsto x$, we factor the sum via the Chinese remainder theorem and note that the local factors at $p$ are either of modulus $p-1$ when $p|\ell-k$, or of modulus 1 otherwise since $\chi$ is quadratic.  

Now apply the identity 
$(k-\ell,q)=\sum_{e\mid q,\ e\mid k-\ell}\phi(e)$ to give 
\[
\sum_{e|q}\phi(e)\sum_{\substack{k,\ell\\k\equiv \ell\bmod e}}|b_kb_\ell|
=
\sum_{e|q}\phi(e) \sum_{a\bmod e}\bigg(\sum_{k\equiv a\,(e)}|b_k|\bigg)^2
\]
  For each $e\mid q$, Cauchy's
inequality in residue classes modulo $e$ gives
\[
 \sum_{a\bmod e}\bigg(\sum_{k\equiv a\,(e)}|b_k|\bigg)^2
 \ll (K/e+1)\sum_k|b_k|^2.
\]
Summing over $e\mid q$ and using $\sum_{e|q}\phi(e)=q$, $\sum_{e|q}\phi(e)/e\leqs d(q)$ and the divisor bound proves \eqref{eq:translate-L2}.
\end{proof}

Applying these in the Cauchy--Schwarz bound \eqref{CS} and bounding $\Lambda(k)^2$ trivially gives 
\begin{align*}
\frac{E(K,M,R)}{KM}
\ll &
 \frac{1}{KM}\cdot q^\eps 
 \bigg(RM+\frac{(RM)^2}{q}\bigg)^{1/2} 
  ((q+K)K)^{1/2} 
 \\
 \ll & {q^\eps}\bigg(\sqrt{\frac{R}{M}}+\frac{R}{q^{1/2}}\bigg)
\end{align*}
since $q\ll K$. Then using $R/q^{1/2}\ll q^{-\delta}$, as before, and $R/M\ll K/T\ll q^{-\eta}$ in this range, gives 
\begin{equation}\label{II}
\frac{E(K,M,R)}{KM}
\ll q^{-\eta/2+\eps}+q^{-\delta+\eps}\ll q^{-\eta/4}
\end{equation}
on taking $\eps$ sufficiently small.

\subsection{Range III: $K\geqs q^{7/3+2\eta}$} We start again from the formula
\[
E(K,M,R)
=
\sum_{m,r}b_mc_r\sum_k a_k\Lambda(k)\chi(k+r\bar m).
\]
If $(k,q)>1$ then we must have $k=p^j$ for some prime $p|q$. Since $a_k$ is supported on a dyadic interval, at most one such prime power can appear in the sum for any given $p>2$. Then bounding trivially, the contribution from $(k,q)>1$ is $\ll \omega(q)\log q RM\ll q^\eps RM$ which on normalising by $KM$ is $\ll q^\eps R/K\ll q^{-1}$ in this range since $R\ll L/T=q^{1/2-\delta}$. Since this is negligible, we may assume that $(k,q)=1$ and we retain this condition implicitly in the $a_k$.  

Now, for a Dirichlet character $\psi$ modulo $q$ consider the generalised Jacobi sum 
\[
T_\psi(y)=\sideset{}{^*}\sum_{u\bmod q}\chi(u+y)\ol{\psi(u)}. 
\]
Then by orthogonality
\[
\frac{1}{\phi(q)}\sum_{\psi\bmod q}T_\psi(y)\psi(k)
=
\sideset{}{^*}\sum_{u\bmod q}\chi(u+y)\mathds{1}_{u\equiv k\bmod q}
=
\chi(k+y).
\]
Applying this identity we acquire
\begin{equation}\label{E III}
E(K,M,R)
=
\frac{1}{\phi(q)}\sum_{\psi\bmod q}\bigg(\sum_{m,r}b_mc_rT_\psi(r\bar m)\bigg)\bigg(\sum_k a_k\Lambda(k)\psi(k)\bigg).
\end{equation}
Roughly speaking, $T_\psi(r\bar m)$ is generically $\ll q^{1/2}$ whilst via zero density estimates we have $\sum_{\psi}\sum_k\cdots \ll Kq^\eps$ giving a total contribution $\ll q^{1/2+\eps}KMR/\phi(q)$. On normalising by $KM$ this is $\ll R/q^{1/2-\eps}\ll q^{-\delta/2}$, as required. These bounds are given by the following two Lemmas.  

\begin{lem}
\label{lem:Jacobi}
Suppose $q$ is odd and square-free and decompose
$\psi=\prod_{p\mid q}\psi_p$. Let
\begin{equation}
 s_\psi=\prod_{\substack{p\mid q\\\psi_p=\chi_p}}p,
 \qquad q_\psi=q/s_\psi,
 \qquad \lambda_\psi=(\chi\bar\psi)|_{q_\psi}.
 \label{eq:s-psi}
\end{equation}
Then
\begin{equation}
 T_\psi(r\bar m)=
 J_{q_\psi}(\psi)\lambda_\psi(r)
 \overline{\lambda_\psi(m)}c_{s_\psi}(r),
 \label{eq:Jacobi-exact}
\end{equation}
where $c_s$ is the Ramanujan sum and
\begin{equation}
 |J_{q_\psi}(\psi)|\leqs\sqrt{q_\psi}.
 \label{eq:Jacobi-bound}
\end{equation}
For $\psi=\psi_0$ or $\chi$, we have $|J_{q_\psi}(\psi)|\leqs 1$.  Thus, for coefficients $b_m$, $c_r$ of modulus at
most $1$,
\begin{equation}
\begin{aligned}
 \frac1{\phi(q)}
 \left|\sum_{r\asymp R}\sum_{m\asymp M}
 c_rb_mT_\psi(r\bar m)\right|
 &\ll q^\eps RM
 \begin{cases}
 q^{-1/2},&\psi\notin\{\psi_0,\chi\},\\
 q^{-1},&\psi\in\{\psi_0,\chi\}.
 \end{cases}
\end{aligned}
\label{eq:Jacobi-coefficients}
\end{equation}
The bound \eqref{eq:Jacobi-coefficients} also holds for fundamental discriminant conductors with a factor $4$ or $8$.
\end{lem}

\begin{proof}The proof shares similarities with Lemma \ref{lem:completion}, and so we may omit some details at times. Recall that 
\[
T_\psi(r\bar m)=\sideset{}{^*}\sum_{u\bmod q}\chi(u+r\bar m)\ol{\psi(u)}. 
\]
By the Chinese Remainder Theorem everything here factors prime by prime.  If $p\nmid r$, we scale $u$ by
$r\bar m$ so that the local sum is a Jacobi sum 
\[
\sideset{}{^*}\sum_{u\bmod p}\chi_p(u+1)\ol{\psi_p(u)}
\]
times
$(\chi_p\bar\psi_p)(r\bar m)$.  If $\psi_p\ne\chi_p$, the Jacobi factor has
absolute value at most $\sqrt p$ by the well-known beta function analogy: for characters $\chi_1,\chi_2$ modulo $m$,
\[
J(\chi_1,\chi_2):=\sum_{u\bmod m} \chi_1(u)\chi_2(1-u)=\frac{\mc{G}(\chi_1)\mc{G}(\chi_2)}{\mc{G}(\chi_1\chi_2)}
\]
where $\mc{G}(\chi_j)=\sum_{a=1}^m\chi_j(a)e(a/m)$ is the Gauss sum. 

If $\psi_p=\chi_p$, the local factor is
$-1$ when $p\nmid r$ and $p-1$ when $p\mid r$, which is exactly $c_p(r)$.
If $p\mid r$ but $\psi_p\ne\chi_p$, both sides of
\eqref{eq:Jacobi-exact} vanish.  This proves \eqref{eq:Jacobi-exact} and
\eqref{eq:Jacobi-bound}.

Now, when summing over $m,r$ we use 
\begin{equation*}
 \sum_{0<|r|\leqs R}|c_s(r)|\leqs
 \sum_{0<|r|\leqs R}(r,s)\ll Rq^\eps.
\end{equation*}
After division by $\phi(q)=q^{1-o(1)}$, the generic bound follows from
$\sqrt{q_\psi}\leqs\sqrt q$.  For $\psi_0$ and $\chi$, all remaining local
Jacobi factors have magnitude $1$, giving the additional square-root saving
in \eqref{eq:Jacobi-coefficients}.  

If $q$ contains a factor of  $2^\nu$, we acquire an additional factor of 
\[
\sum_{u\bmod 2^\nu}\chi_{2^\nu}(u+r\bar m)\ol{\psi_{2^\nu}(u)}
\]
in the factorisation of $T_\psi(r\bar m)$. Obviously, this only changes the bound up to a constant. 
\end{proof}

\begin{lem}
\label{lem:prime-L1}
Let $V$ be a
smooth function supported in $[1,2]$ and suppose that
$\|V^{(j)}\|_\infty\ll_j q^{\eps}$.  If
\begin{equation*}
 K\geqs q^{7/3+2\eta},
\end{equation*}
then
\begin{equation*}
 \sum_{\substack{\psi\bmod q\\\psi\notin\{\psi_0,\chi\}}}
 \left|\sum_n\Lambda(n)\psi(n)V(n/K)\right|
 \ll_{\eta,\eps}Kq^\eps
\end{equation*}
where $\psi_0$ is the principal character. 

\end{lem}

\begin{proof}
We first replace the sum by that involving the (primitive) inducing characters, so as to apply zero-density results later. Let $\psi^*$ be the primitive character inducing $\psi$, and let
$d_\psi$ be its conductor.  If $\psi\ne\psi_0$, then $\psi^*$ is nonprincipal also. 
The difference between the two sums is 
\begin{align*}
 &\sum_n\Lambda(n)\psi(n)V(n/K)
 -\sum_n\Lambda(n)\psi^*(n)V(n/K)\\
 &\qquad=-\sum_{\substack{p\mid q\\p\nmid d_\psi}}
 \sum_{\substack{r\geqs1\\K\leqs p^r\leqs 2K}}
 (\log p)\psi^*(p^r)V(p^r/K).
\end{align*}
For a fixed prime $p$, at most two powers $p^r$ lie in $[K,2K]$ and consequently
\[
 \left|\sum_n\Lambda(n)
       (\psi(n)-\psi^*(n))V(n/K)\right|
 \ll \|V\|_\infty\sum_{p\mid q}\log p
 \ll q^{\eps}.
\]
Summing over the $\psi$ modulo $q$, the total difference is
\begin{equation*}
 O\bigl(q^{1+\eps}\bigr)=o(K).
\end{equation*}
Thus, it suffices to consider the sum
\[
 \sum_{\substack{\psi\bmod q\\\psi\notin\{\psi_0,\chi\}}}
\bigg|\sum_n\Lambda(n)\psi^*(n)V(n/K)\bigg|.
\]

We now compute the inner sum via the explicit formula. Write
\[
 \widetilde V(s)=\int_0^\infty V(x)x^{s-1}\,dx,
 \qquad
 \mathcal V_j=1+\max_{0\leqs \ell\leqs j}
                    \|V^{(\ell)}\|_\infty.
\]
From the compact support of $V$, integration by parts gives the rapid decay bounds
\begin{equation}
 |\widetilde V(\sigma+it)|
 \ll_j \mathcal V_j(1+|t|)^{-j}
 \qquad (-1\leqs\sigma\leqs2).
 \label{eq:Mellin-decay}
\end{equation}
Now, by Mellin inversion we have 
\begin{equation}
 \sum_n\Lambda(n)\psi^*(n)V(n/K)
 =\frac{1}{2\pi i}\int_{(2)}
   -\frac{L'}{L}(s,\psi^*)\widetilde V(s)\,K^sds.
 \label{eq:prime-Mellin-inversion}
\end{equation}
Shifting the contour to $\Re s=-1/2$ we encounter poles at the nontrivial zeros,
and possibly trivial zeros including the zero
at $s=0$ that occurs for an even nonprincipal character.  By
\eqref{eq:Mellin-decay}, the trivial-zero residues and the new vertical
integral contribute
\[
 O_B\bigl(\mathcal V_{B+2}\log(2qK)\bigr).
\]
Thus,
\begin{equation}
 \frac1K\left|\sum_n\Lambda(n)\psi^*(n)V(n/K)\right|
 \ll_B q^\eps
 \bigg(\sum_{\rho_{\psi^*}}
 K^{\beta-1}(1+|\gamma|)^{-B}
 +\frac{\log(2qK)}K\bigg).
 \label{eq:prime-explicit}
\end{equation}

We first estimate the tails of the zero sums and for this we need an estimate for the number of zeros. If $F$ is any
nonnegative function of a primitive character, then
\begin{equation*}
 \sum_{\psi\bmod q}F(\psi^*)
 =\sum_{d\mid q}\ \sideset{}{^*}\sum_{\xi\bmod d}F(\xi)
\end{equation*}
where the inner sum is over primitive characters modulo $d$. The
standard individual zero count
\[
 N_\xi(Y):=\#\{\rho_\xi:|\Im\rho_\xi|\leqs Y\}
 \ll Y\log(2dY) \qquad (Y\geqs1)
\]
therefore yields
\begin{equation}
 \mathcal N_q(Y):=
 \sum_{\psi\bmod q}
 \#\{\rho_{\psi^*}:|\gamma|\leqs Y\}
 \ll qY\log(2qY).
 \label{eq:aggregate-zero-count}
\end{equation}

Now, consider the zeros with $|\gamma|>H$ where  
\[
H=q^b,\qquad b>0
\] 
with $b$ to be determined.
Using
\eqref{eq:aggregate-zero-count} and taking $B$ large enough in terms of (any given positive) $b$ we obtain
\begin{align*}
 &\sum_{\psi\bmod q}\sum_{|\gamma|>H}
 K^{\beta-1}(1+|\gamma|)^{-B}
 \ll qH^{1-B}\log(2qH)
 \ll q^{-2}\log(2qH)
\end{align*}
since $\beta<1$. This is $o(1)$, with a fixed power saving.

For the zeros with $\beta<1/2$ we have, from
\eqref{eq:aggregate-zero-count} and the assumption $K\geqs q^{7/3+2\eta}$,
\begin{equation*}\begin{split}
 \sum_{\psi\bmod q}
 \sum_{\substack{|\gamma|\leqs H\\\beta<1/2}}K^{\beta-1}
 &\leqs K^{-1/2}\mathcal N_q(H) \\
 &\ll qHK^{-1/2}\log(2qH) \\
 &\ll q^{-1/6+b-\eta}\log(2qH)=o(1)
 \label{eq:left-zero-range}
\end{split}
\end{equation*}
provided that $b\leqs1/6$.

It remains to consider $\beta\geqs1/2$ and $|\gamma|\leqs H$.  Set
\[
 A(\sigma)=\sum_{\psi\bmod q}N(\sigma,H,\psi^*).
\]
The elementary identity
\[
 K^{\beta-1}=K^{-1/2}
 +\log K\int_{1/2}^{\beta}K^{\sigma-1}\,d\sigma
 \qquad (\beta\geqs1/2)
\]
gives, after summing one zero at a time,
\begin{align}
{\sum_{\psi\bmod q}}
   {\sum_{\substack{\rho_{\psi^*}\\
          \beta\geqs1/2,\ |\gamma|\leqs H}}}K^{\beta-1}
 &=K^{-1/2}A(1/2)
 +\log K\int_{1/2}^{1}K^{\sigma-1}A(\sigma)\,d\sigma.
 \label{eq:layer-cake-identity}
\end{align}
Now, Proposition~\ref{prop:density} gives 
\begin{equation}
 A(\sigma)
 =
 \sum_{d|q}\sideset{}{^*}\sum_{\xi\bmod d}N(\sigma,H,\xi)
 \ll (qH)^{7(1-\sigma)/3+\eps}
 \label{eq:density-in-prime-proof}
\end{equation}
on applying the divisor bound $d(q)\ll q^\eps$. 

The first term on the right of \eqref{eq:layer-cake-identity} is thus 
\[
\ll (q^{1+b})^{7/6+\eps}/q^{7/6+\eta}
\]
which is $o(1)$ provided $b<6\eta/7$, say. Using $K\geqs q^{7/3+2\eta}$ again and $H=q^b$ the integral is 
\[
\ll q^{\eps}\int_{1/2}^1 q^{-(\tfrac73+2\eta)(1-\sigma)+\tfrac73(1+b)(1-\sigma)}d\sigma
\]
which is $\ll q^\eps$, again provided $b<6\eta/7$. Thus, in total we find that 
\begin{align}
 &\frac1K
 \sum_{\substack{\psi\bmod q\\\psi\notin\{\psi_0,\chi\}}}
 \left|\sum_n\Lambda(n)\psi^*(n)V(n/K)\right|
\ll
 q^{\eps}
 +q^{1+\eps }\frac{\log(2qK)}K\ll q^\eps.
 \label{eq:primitive-prime-aggregate}
\end{align}
The result then follows on taking $0<b<\min(\tfrac16,\tfrac{6\eta}{7})$.
\end{proof}

\begin{rem} Note that the restriction on the size of $K$ in this regime is directly related to the exponent $7/3$ in the zero density estimate. If we had exponent $\alpha$, so that the Density Hypothesis gives $\alpha=2$, for example, we would be able to take $K>q^{\alpha+\eta}$. Now, the requirement of having a slight power saving over $T$ in the upper bound of   Range II, coupled with the requirement that Ranges II and III  have non-empty overlap dictates that, up to $\eps$'s and $\eta$'s, $T> q^\alpha$. Therefore, on taking $T\approx q^\alpha$ we see that our arguments allow a resonator of length
\[
L= Tq^{1/2-\delta}\approx T^{1+\tfrac{1}{2\alpha}}.
\]
\end{rem}

Combining Lemmas \ref{lem:Jacobi} and \ref{lem:prime-L1} in \eqref{E III}, and bounding the sum over $k$ trivially when $\psi=\psi_0$ or $\chi$, gives
\begin{equation}\label{III}
\frac{E(K,M,R)}{KM}\ll \frac{1}{KM}\cdot q^\eps\frac{RMK}{q^{1/2}}\ll q^{-\delta/2},
\end{equation}
 as required. Combining the power savings from all three ranges: \eqref{I},\eqref{II} and \eqref{III} gives Proposition \ref{off diag prop}.


\end{document}